\documentclass[12pt]{article}
\usepackage{amsmath, amsfonts, amsthm, amssymb, epsfig}

\newtheorem{theorem}{Theorem}
\newtheorem{lemma}{Lemma}

\newtheorem{proposition}{Proposition}

\begin{document}
\title{{\bf Properties of proper rational numbers}}
\author{Konstantine Zelator\\
Mathematics, Statistics, and Computer Science\\
212 Ben Franklin Hall\\
Bloomsburg University of Pennsylvania\\
400 East 2nd Street\\
Bloomsburg, PA  17815\\
USA\\
and\\
P.O. Box 4280\\
Pittsburgh, PA  15203\\
kzelator@.bloomu.edu\\
e-mails: konstantine\underline{\ }zelator@yahoo.com}

\maketitle

\section{Introduction}

The set of rational numbers can be thought of as the disjoint union of two of
its main subsets:  the set of integers and the set of proper rationals.

\vspace{.15in}

\noindent {\bf Definition 1:}  A proper rational number is a rational number
which is not an integer.

The aim of this work is simple and direct.  Namely, to explore some of the
basic or elementary properties of the proper rationals.

We will make use of the standard notation $(u,w)$ denoting the greatest common
divisor of two integers $u$ and $w$.  Also, the notation $u|w$ to denote that
$u$ is a divisor of $w$.

\begin{proposition}  Let $r$ be a proper rational number.  Then $r$ can be
  written in the form, $r=\frac{c}{b}$ where $c$ and $b$ are relatively prime
  integers; $(c,b)=1$, and with $b \geq 2$.
\end{proposition}

\begin{proof}  Since $r$ is a proper rational, it cannot, by definition, be
  zero.  Hence $r = \frac{A}{B}$, for some positive integers $A$ and $B$; if
  $r > 0$.  If, on the other hand, $r < 0$, then $r = -\frac{A}{B}$,  $A,B$
  being positive integers.  Let $d = (A,B)$, then $A =da$, $B = db$, for relative prime positive integers $a$ and $b$.  We have, $r = \frac{A}{B} =
  \frac{da}{db} = \frac{a}{b}$, for $ r > 0$.  Clearly, $b$ cannot equal 1,
  for then $r$ would equal $a$, an integer, contrary to the fact that $r$ is a
  proper rational.  Hence, $b \geq 2$.

If, on the other hand, $r < 0$, $r = \frac{-A}{B} = \frac{-a}{b} =
\frac{c}{b}$ with $c = -a$, and $b \geq 2$. \end{proof}

\vspace{.15in}

\noindent {\bf Definition 2:} A proper rational number $r$ is said to be in
standard form if it is written as $r = \frac{c}{b}$, where $c$ and $b$ are
relatively prime integers and $b \geq 2$

\section{The reciprocal of a proper rational}

We state the following result without proof.  We invite the interested reader tofill in the details.

\begin{theorem}  Let $r = \frac{c}{b}$ be a proper rational in standard form.  

\begin{enumerate}
\item[(i)] If $|c|=1$, the reciprocal $\frac{1}{r} = \frac{b}{c}$ is an
  integer.

\item[(i)] If $c \geq 2$, the reciprocal $\frac{1}{r}$ is a proper positive
  rational.

\item[(iii)]  If $|c|\geq 2$ and $c < 0$, the reciprocal $\frac{1}{r}$ is
  negative proper rational with the standard form being $\frac{1}{r} =
  \frac{d}{|c|}$, where $d=-b$.
\end{enumerate}
\end{theorem}

\section{An obvious property}

Is the sum of a proper rational and an integer always a proper rational?  The
answer is a rather obvious yes.

\begin{theorem}  Suppose that $r = \frac{c}{b}$ is a proper rational in
  standard form; and $d$ and integer.  Then the sume $r+d$ is a proper
  rational.
\end{theorem}

\begin{proof}  If, to the contrary, $r+d=i$, for some $i \in {\mathbb Z}$, then $r
  = i-d$, an integer contradicting the fact that $r$ is a proper
  rational. \end{proof} 

\section{A lemma from number theory}

We will make repeated use of the very well known, and important, lemma below.
For a proof of this lemma, see reference \cite{1}.  It can be found in just
about every elementary number theory book.

\begin{lemma}  (Euclid's lemma)

\begin{enumerate}
\item[(i)]  (Standard version) Let $m,n,k$ be positive integers such that $m$
  is a divisor of the product $n\cdot k$; and suppose that $(m,n)=1$.  Then $m$ is a
  divisor of $k$.

\item[(ii)]  (Extended version)  Let $m,n,k$ be non-zero integers such that
  $m|nk$ and $(m,n)=1$.  Then $m|k$.
\end{enumerate}
\end{lemma}

\section{A slightly less obvious property}

When is the product of a proper rational with an integer, an integer? A proper
rational?

\begin{theorem}  Let $r = \frac{c}{b}$ be a proper rational in standard form
  and $i$ an integer.

\begin{enumerate}
\item[(a)]  The product $r\cdot i$ is an integer if, and only if, $b|i$.

\item[(b)]  The product $r \cdot i$ is a proper rational if, and only if, $b$
  is not a divisor of $i$.
\end{enumerate}
\end{theorem}

\begin{proof} (b) This part is logically equivalent to
  part (a). 

\noindent (a)  If $b$ divides $i$, then $i = b \cdot q$, an integer.

\noindent so we have $r \cdot i = \frac{cbq}{b} = c\cdot q$, an integer.  Now
the converse.  Suppose that $r\cdot i$ is an integer $t$:  $r\cdot i = t$,
which yields,

\begin{equation}
c \cdot i = b \cdot t \label{E1}
\end{equation}

\noindent Since $r$ is a proper rational, $(b,c)=1$ by defintion.  Equation
(\ref{E1}) shows that $b|c\cdot i$; and since $(b,c) = 1$.  Lemma 1 implies
that $b$ must divide $i$.  We are done.  
\end{proof}

\section{The sum of two proper rationals}

An interesting equation arises.  When is the sum of two proper rationals also
a proper rational?  When is it an integer?  There is no obvious answer here.

\begin{theorem}  Let $r_1 = \frac{c_1}{b_1}$ and $r_2 = \frac{c_2}{b_2}$ be
  proper rationals in standard form.  Then,

\begin{enumerate}
\item[(i)]  The sum $r_1 + r_2$ is an integer if, and only if, $b_1 = b_2$ and
  $b_1$ is a divisor of the sum $c_1+c_2$.

\item[(ii)]  The sum $r_1 + r_2$ is a proper rational if, and only if,
  either $b_1 \neq b_2$ or $b_1 = b_2$ but with $b_1$ not being a divisor of
  $c_1 + c_2$.
\end{enumerate}
\end{theorem}

\begin{proof}  (ii) This part is logically  equivalent to part (i).

\noindent (i) If $b_1=b_2$ and $b_1|(c_1+c_1)$, then $r_1 + r_2 =
\frac{c_1+c_2}{b_1}$, is obviously an integer.  Next, let us prove the
converse statement.  

Suppose that $r_1+r_2 = i$, an integer.  Some routine algebra produces

\begin{equation}
c_1b_2 + c_2b_1 = i b_1 b_2 \label{E2} 
\end{equation}

\noindent or equivalently

\begin{equation}
c_1b_2 = b_1(ib_2-c_2) \label{E3}.
\end{equation}

According to (\ref{E3}), $b_1|c_1b_2$; and since $(b_1,c_1)=1$, Lemma 1
implies that $b_1|b_2$.  A similar argument, using equation (\ref{E2}), once
more establishes that $b_2|b_1$.  Clearly, since the two positive integers
$b_1$ and $b_2$ are divisors of each other, they must be equal; $b_1 =b_2$
(an easy exercise in elementary number theory).  From $b_1 = b_2$ and
(\ref{E2}), we obtain $c_1+c_2 = i\cdot b_i$; and thus it is clear that
$b_1|(c_1+c_2)$. \end{proof}

\section{The product of two proper rationals}

\begin{theorem} Let $r_1 = \frac{c_1}{b_1}$ and $r_2 = \frac{c_2}{b_2}$ be
  proper rationals in standard form.

\begin{enumerate}
\item[(a)]The product $r_1r_2$ is an integer if, and only if, $b_1|c_2$ and
  $b_2|c_1$.

\item[(b)] The product $r_1r_2$ is a proper rational if, and only if, $b_1$ is
  not a divisor of $c_2$; or $b_2$ is not a divisor of $c_1$.
\end{enumerate}
\end{theorem}

\begin{proof} (b) This part is logically equivalent to part (a).

\noindent (a) Suppose that $b_1|c_2$ and $b_2|c_1$; then $c_2=b_1a$ and
$c_1=b_2d$ where $a$ and $d$ are (non-zero) integers.

We have $r_1r_2 = \frac{c_1c_2}{b_1b_2} = \frac{adb_1 b_2}{b_1b_2} = ad$, an
integer.  

Conversely, suppose that $r_1r_2=i$, an integer.  Then

\begin{equation}
c_1c_2 = ib_1b_2 \label{E4}
\end{equation}

Since $(b_1,c_1) = 1 = (c_2,b_2)$, (\ref{E4}), in conjunction with Lemma 1,
imply that $b_1|c_2$ and $b_2|c_1$.  We are done.  \end{proof}

\section{ One more result and its corollary}

In Theorem 4 part (i), gives us the precise conditions for the sum of two
proper rationals to be an integer. Likewise, Theorem 5 part (a) gives us the
exact conditions for the product of two proper rationals to be an
integer. Naturally, the following question arises.  Can we find two proper
rational numbers whose sum is an integer; and also whose product is an
integer?  Theorem 7 provides an answer in the negative. Theorem 7 is a direct
consequence of Theorem 6 below.

\begin{theorem} If both the sum and the product of two rational numbers are 
integers, then so are the two rationals, integers.
\end{theorem}

\begin{proof} Let $r_1,r_2$ be the two rationals, and suppose that 

\begin{equation} \left\{\begin{array}{rcl} r_1+r_2 & = & i_1\\
r_1r_2 & = & i_2\\
i_1,i_2 \in {\mathbb Z} \end{array}\right\} \label{E5}
\end{equation}

\noindent If either of $r_1,r_2$ is an integer, then the first equation in
(\ref{E5}) implies that the other one is also an integer.  
So  we are done in this case.  So, assume that neither of $r_1,r_2$ is an
integer; which means that they are both proper rationals.  Let then $r_1 =
\frac{c_1}{b_1},\ r_2 = \frac{c_2}{b_2}$ be the standard forms of $r_1$ and
$r_2$.  That is, $(c_1,b_1) = 1 = (c_2,b_2),\ b_1 \geq 2,\ b_2 \geq 2$ and, of
course, $c_1c_2 \neq 0$.

Combining this information with (\ref{E5}), we get 

\begin{equation} \left\{ \begin{array}{rcl} c_1b_2 + c_2b_1 & = & i_1 b_1b_2\\
c_1c_2 & = & i_2b_1b_2 \end{array}\right\} \label{E6}
\end{equation}

From the first equation in (\ref{E6}) we obtain

$$c_1b_2 = b_1(i_1b_2-c_2),
$$

\noindent which shows that $b_1|c_1b_2$.

This, combined with $(c_1,b_1)=1$ and Lemma 1 allow us to deduce that
$b_1|b_2$.  Similarly, using the first equation in (\ref{E6}), we infer that
$b_2|b_1$ which implies $b_1=b_2$.  Hence, the second equation of (\ref{E6})
gives,

\begin{equation}
c_1c_2 = i_2 b_1^2 \label{E7}
\end{equation}

By virtue of $(b_1,c_1)=(b_1,c_2) = 1$, equation (\ref{E7}) implies $b_1=1;\
b_1=b_2 = 1$.  Therefore $r_1$ and $r_2$ are integers.  \end{proof}

We have the immediate corollary.

\begin{theorem}  There exist no two proper rationals both of whose sum and
  product are integers.\end{theorem}

\section{A closing remark}

Theorem 7 can also be proved by using the well known Rational Root Theorem for
polynomials with integer coefficients.  The Rational Root Theorem implies that
if a monic (i.e., leading coefficient is $1$) polynomial with integer
coefficients has a rational root that root must be an integer.   Every
rational of such a monic polynomia must be an integer  (equivalently, each of
its real roots, if any, must be either an irrational number or an integer).
Thus, in our case, the rational numbers $r_1$ and $r_2$ are the roots of the
monic trinomial, $t(x)= (x-r_1)(x-r_2) = x^2 - i_1x + i_2$; a monic quadratic
polynomial with integer coefficients $-i_1$ and $i_2$.  Hence, $r_1$ and $r_2$
must be integers.

For more details, see reference \cite{1}.


\begin{thebibliography}{99}
\bibitem{1}{1} Kenneth H. Rose, {\it Elementary Numbers Theory and Its
  Applications}, fifth edition, 2005, Pearson-Addison-Wesley.

For Lemma 1 (Lemma 3.4 in the above book), see page 109 for the Rational Root
Theorem, see page 115.
\end{thebibliography}
\end{document}